\documentclass[12pt,reqno]{article}
\usepackage{amssymb,amscd,amsmath,amsthm,color}
\newtheorem{thm}{Theorem}[section]
\newtheorem{prop}[thm]{Proposition}
\newtheorem{lemma}[thm]{Lemma}
\newtheorem{cor}[thm]{Corollary}

\newtheorem{remark}[thm]{Remark}

\newtheorem{problem}[thm]{Problem}
\numberwithin{equation}{section}

\def\bR{\mathbb{R}}
\def\bN{\mathbb{N}}
\def\bM{\mathbb{M}}
\def\Tr{\mathrm{Tr}\,}
\def\diag{\mathrm{diag}}
\def\BS{\mathrm{BS}}

\begin{document}
\baselineskip=16pt
\allowdisplaybreaks

\centerline{\LARGE Log-majorization related to R\'enyi divergences}
\bigskip
\bigskip
\centerline{\Large Fumio Hiai\footnote{{\it E-mail address:} hiai.fumio@gmail.com}}

\medskip
\begin{center}
$^1$\,Tohoku University (Emeritus), \\
Hakusan 3-8-16-303, Abiko 270-1154, Japan
\end{center}

\medskip
\begin{abstract}
For $\alpha,z>0$ with $\alpha\ne1$, motivated by comparison between different kinds of R\'enyi
divergences in quantum information, we consider log-majorization between the matrix functions
\begin{align*}
P_\alpha(A,B)&:=B^{1/2}(B^{-1/2}AB^{-1/2})^\alpha B^{1/2}, \\
Q_{\alpha,z}(A,B)&:=(B^{1-\alpha\over2z}A^{\alpha\over z}B^{1-\alpha\over2z})^z
\end{align*}
of two positive (semi)definite matrices $A,B$. We precisely determine the parameter $\alpha,z$
for which $P_\alpha(A,B)\prec_{\log}Q_{\alpha,z}(A,B)$ and $Q_{\alpha,z}(A,B)\prec_{\log}
P_\alpha(A,B)$ holds, respectively.


\end{abstract}

\section{Introduction}

For each $n\in\bN$ we write $\bM_n$ for the $n\times n$ complex matrices and $\bM_n^+$ for the
positive semidefinite $n\times n$ matrices. We write $B>0$ if $B\in\bM_n$ is positive definite.

Recall that for $X,Y\in\bM_n^+$, the \emph{log-majorization} $X\prec_{\log}Y$ means that
$$
\prod_{i=1}^k\lambda_i(X)\le\prod_{i=1}^k\lambda_i(Y),\qquad k=1,\dots,n
$$
with equality for $k=n$ (i.e., $\det X=\det Y$), where $\lambda_1(X)\ge\dots\ge\lambda_n(X)$
are the eigenvalues of $X$ in decreasing order counting multiplicities. As is well-known,
$X\prec_{\log}Y$ implies the \emph{weak majorization} $X\prec_wY$, i.e.,
$\sum_{i=1}^k\lambda_i(X)\le\sum_{i=1}^k\lambda_i(Y)$ for $k=1,\dots,n$. The latter is
equivalent to that $\|X\|\le\|Y\|$ holds for every unitarily invariant norm $\|\cdot\|$.
See, e.g., \cite{An,Bh,Hi2,MOA} for generalities on majorization theory for matrices. 

Let $\alpha,z>0$ with $\alpha\ne1$, and let $A,B\in\bM_n^+$ with $B>0$. In the present paper
we are concerned with the following two-variable matrix functions
\begin{align*}
P_\alpha(A,B)&:=B^{1/2}(B^{-1/2}AB^{-1/2})^\alpha B^{1/2}, \\
Q_{\alpha,z}(A,B)&:=(B^{1-\alpha\over2z}A^{\alpha\over z}B^{1-\alpha\over2z})^z.
\end{align*}
Two special versions of $Q_{\alpha,z}$ are
\begin{align*}
Q_\alpha(A,B)&:=Q_{\alpha,1}(A,B)=B^{1-\alpha\over2}A^\alpha B^{1-\alpha\over2}, \\
\widetilde Q_\alpha(A,B)&:=Q_{\alpha,\alpha}(A,B)
=(B^{1-\alpha\over2\alpha}AB^{1-\alpha\over2\alpha})^\alpha.
\end{align*}
Our motivation to consider these functions comes from different types of R\'enyi divergences
that have recently been developed in quantum information. The conventional (or standard)
\emph{$\alpha$-R\'enyi divergence} (due to Petz \cite{Pe}) is
$$
D_\alpha(A\|B):={1\over\alpha-1}\log{\Tr Q_\alpha(A,B)\over\Tr A},
$$
the \emph{sandwiched $\alpha$-R\'enyi divergence} \cite{MDSFT,WWY} is
$$
\widetilde D_\alpha(A\|B):={1\over\alpha-1}\log{\Tr\widetilde Q_\alpha(A,B)\over\Tr A},
$$
and the so-called \emph{$\alpha$-$z$-R\'enyi divergence} \cite{AD} is
$$
D_{\alpha,z}(A\|B):={1\over\alpha-1}\log{\Tr Q_{\alpha,z}(A,B)\over\Tr A}.
$$
In addition to $D_\alpha$ and $\widetilde D_\alpha$ we define the \emph{maximal
$\alpha$-R\'enyi divergence}
$$
\widehat D_\alpha(A\|B):={1\over\alpha-1}\log{\Tr P_\alpha(A,B)\over\Tr A}.
$$
(For the term ``maximal" here, see Remark \ref{R-8.4}.) See \cite{HM} and references therein
for more background information on quantum divergences.

We note that $P_\alpha(A,B)$ is a special case of \emph{operator perspective} defined
associated with a function $f$ on $(0,\infty)$ by
$$
P_f(A,B):=B^{1/2}f(B^{-1/2}AB^{-1/2})B^{1/2},\qquad A,B\in\bM_n^+,\ A,B>0,
$$
which was studied by Effros and Hansen \cite{EH} and others, with applications to quantum
information. Furthermore, note that when $f$ is a non-negative operator monotone function on
$(0,\infty)$ with $f(1)=1$, $P_f(A,B)$ is nothing but the \emph{operator mean}
$B\,\sigma_f\,A$ associated with $f$ in the Kubo-Ando sense \cite{KA}. In particular, when
$0<\alpha<1$, $P_\alpha(A,B)=B\,\#_\alpha\,A$, the \emph{weighted geometric mean} (first
introduced by Pusz and Woronowicz \cite{PW} in the case $\alpha=1/2$).

The inequality $\widetilde D_\alpha(A\|B)\le D_\alpha(A\|B)$ is well-known as a consequence
of Araki's log-majorization \cite{Ar} (see also \cite{AH}); indeed, $Q_{\alpha,z}(A,B)$ is
monotone decreasing in $z>0$ in the log-majorization order. However, the comparison between
$\widehat D_\alpha$ and $D_{\alpha,z}$ (in particular, $D_\alpha$) has not fully been
investigated so far, which motivate us to consider the log-majorization between $P_\alpha$
and $Q_{\alpha,z}$. In this paper we present the following theorem which was announced without
proofs in \cite[Remark 4.6]{HM}:

\begin{thm}\label{T-1.1}
Let $A,B\in\bM_n^+$ with $B>0$.
\begin{itemize}
\item[\rm(1)] For $0<\alpha<1$ and $z>0$, $P_\alpha(A,B)\prec_{\log}Q_{\alpha,z}(A,B)$.
\item[\rm(2)] For $\alpha>1$ and $0<z\le\min\{\alpha/2,\alpha-1\}$,
$P_\alpha(A,B)\prec_{\log}Q_{\alpha,z}(A,B)$.
\item[\rm(3)] For $\alpha>1$ and $z\ge\max\{\alpha/2,\alpha-1\}$,
$Q_{\alpha,z}(A,B)\prec_{\log}P_\alpha(A,B)$.
\end{itemize}
In particular, $P_\alpha(A,B)\prec_{\log}Q_\alpha(A,B)$ if $0<\alpha<1$ or $\alpha\ge2$, and
$Q_\alpha(A,B)\prec_{\log}P_\alpha(A,B)$ if $1<\alpha\le2$.
\end{thm}

The paper is organized as follows. In Sections 2 and 3 we prove Theorem \ref{T-1.1}. In
Section 4 we give an example showing that Theorem \ref{T-1.1} is best possible with regard to
the assumptions on the parameters $\alpha,z$, so that Theorem \ref{T-1.1} is completed into
Theorem \ref{T-4.1}. In Section 5 we present the necessary and sufficient conditions on
$\alpha,r,z$ for which $P_{\alpha,r}(A,B)\prec_{\log}Q_{\alpha,z}(A,B)$ and
$Q_{\alpha,z}(A,B)\prec_{\log}P_{\alpha,r}(A,B)$ hold, respectively, where
$P_{\alpha,r}(A,B):=P_\alpha(A^{1/r},B^{1/r})^r$. Moreover, we give a log-majorization
for $P_\alpha$ for $\alpha\ge2$, supplementing Ando-Hiai's log-majorization \cite{AH} for
$P_\alpha$ for $0<\alpha<1$ and its complementary version recently obtained by Kian and Seo
\cite{KS} for $P_\alpha$ for $1<\alpha\le2$. (Note that the negative power $\beta\in[-1,0)$
case in \cite{KS} can be rephrased into the case of $P_\alpha$ for $\alpha=1-\beta\in(1,2]$,
see Section 5.) Applying our log-majorization results, in Sections 6 and 7 we give norm
inequalities for unitarily invariant norms and logarithmic trace inequalities. The norm
inequalities here improve those given in \cite{KS} and the logarithmic trace inequalities here
supplement those given in \cite{AH}. Finally in Section 8 we completely determine the
parameters $\alpha,z$ for which $\widehat D_\alpha(A\|B)\le D_{\alpha,z}(A\|B)$ and
$D_{\alpha,z}(A\|B)\le\widehat D_\alpha(A\|B)$ hold, respectively.

\section{Log-majorization (Part 1)}

First, note that Araki's log-majorization \cite{Ar} (see also \cite{AH}) implies that for every
$\alpha>0$,
\begin{equation}\label{F-2.1}
Q_{\alpha,z'}(A,B)\prec_{\log}Q_{\alpha,z}(A,B)\qquad
\mbox{if $0<z\le z'$}.
\end{equation}
The next proposition is an easy part of log-majorization results between $P_\alpha$ and
$Q_{\alpha,z}$.

\begin{prop}\label{P-2.1}
Let $A,B\in\bM_n^+$ with $B>0$.
\begin{itemize}
\item[\rm(1)] Assume that $0<\alpha<1$. Then for every $z>0$,
$$
P_\alpha(A,B)\prec_{\log}Q_{\alpha,z}(A,B).
$$
\item[\rm(2)] Assume that $1<\alpha\le2$ and $0<z\le\alpha-1$. Then
$$
P_\alpha(A,B)\prec_{\log}Q_{\alpha,z}(A,B).
$$
\item[\rm(3)] Assume that $\alpha>1$ and $z\ge\max\{\alpha/2,\alpha-1\}$. Then
$$
Q_{\alpha,z}(A,B)\prec_{\log}P_\alpha(A,B).
$$
\end{itemize}
\end{prop}

\begin{proof}
(1) Although this is an immediate consequence of well-known Araki's and Ando-Hiai's
log-majorization (see \cite{Ar,AH}), we give a proof for the convenience of the reader. By
continuity we may assume that $A>0$ as well as $B>0$. From the Lie-Trotter formula, letting
$z'\to\infty$ in \eqref{F-2.1} gives
\begin{align}\label{F-2.2}
\exp(\alpha\log A+(1-\alpha)\log B)\prec_{\log}Q_{\alpha,z}(A,B),\qquad z>0.
\end{align}
On the other hand, when $0<\alpha<1$, the log-majorization in \cite{AH} says that
$$
P_\alpha(A,B)=B\#_\alpha A\prec_{\log}(B^p\#_\alpha A^p)^{1/p},\qquad0<p<1.
$$
Letting $p\searrow0$ and using \cite[Lemma 3.3]{HP2} we have
\begin{align}\label{F-2.3}
P_\alpha(A,B)\prec_{\log}\exp(\alpha\log A+(1-\alpha)\log B).
\end{align}
Combining \eqref{F-2.2} and \eqref{F-2.3} implies the asserted log-majorization.

(2)\enspace
By continuity we may assume that $A>0$ as well as $B>0$. The proof below is an easy application
of the standard anti-symmetric tensor power technique (see, e.g., \cite{AH}). To show that
$P_\alpha(A,B)\prec_{\log}Q_{\alpha,z}(A,B)$, it suffices to prove that
$$
\|P_\alpha(A,B)\|_\infty\le\|Q_{\alpha,z}(A,B)\|_\infty,
$$
where $\|\cdot\|_\infty$ denotes the operator norm. Due to the positive homogeneity in $A,B$
of both $P_\alpha$ and $Q_{\alpha,z}$ (i.e.,
$P_\alpha(\lambda A,\lambda B)=\lambda P_\alpha(A,B)$ for $\lambda>0$ and similarly for
$Q_{\alpha,z}$), it also suffices to prove that
\begin{equation}\label{F-2.4}
Q_{\alpha,z}(A,B)\le I\ \,\implies\ \,P_\alpha(A,B)\le I.
\end{equation}
Here recall the identity
\begin{equation}\label{F-2.5}
(B^{-1/2}AB^{-1/2})^\alpha=B^{-1/2}A^{1/2}(A^{1/2}B^{-1}A^{1/2})^{\alpha-1}A^{1/2}B^{-1/2},
\end{equation}
as seen from the well-known equality
\begin{equation}\label{F-2.6}
Xf(X^*X)=f(XX^*)X
\end{equation}
for every $X\in\bM_n$ and every continuous function $f$ on an interval containing the
eigenvalues of $X^*X$ (the proof is easy by approximating $f$ by polynomials). Therefore,
\begin{align}\label{F-2.7}
P_\alpha(A,B)=A^{1/2}(A^{1/2}B^{-1}A^{1/2})^{\alpha-1}A^{1/2},
\end{align}
so that for \eqref{F-2.4} it suffices to prove that
$$
B^{1-\alpha\over2z}A^{\alpha\over z}B^{1-\alpha\over2z}\le I
\ \,\implies\ \,A^{1/2}(A^{1/2}B^{-1}A^{1/2})^{\alpha-1}A^{1/2}\le I,
$$
or equivalently,
\begin{equation}\label{F-2.8}
A^{\alpha\over z}\le B^{\alpha-1\over z}
\ \,\implies\ \,(A^{1/2}B^{-1}A^{1/2})^{\alpha-1}\le A^{-1}.
\end{equation}
Now, assume that $1<\alpha\le2$ and $0<z\le\alpha-1$, and that
$A^{\alpha\over z}\le B^{\alpha-1\over z}$. Since $0<z/(\alpha-1)\le1$,
$$
B^{-1}=(B^{\alpha-1\over z})^{-{z\over\alpha-1}}
\le(A^{\alpha\over z})^{-{z\over\alpha-1}}=A^{-{\alpha\over\alpha-1}},
$$
and hence
$$
A^{1/2}B^{-1}A^{1/2}\le A^{1/2}A^{-{\alpha\over\alpha-1}}A^{1/2}=A^{-{1\over\alpha-1}}.
$$
Since $0<\alpha-1\le1$, we have
$$
(A^{1/2}B^{-1}A^{1/2})^{\alpha-1}\le(A^{-{1\over\alpha-1}})^{\alpha-1}=A^{-1},
$$
proving \eqref{F-2.8}.

(3)\enspace
As in the proof of (2) we may assume that both $A,B$ are positive definite, and prove the
implication opposite to \eqref{F-2.4}. In the present case, similarly to the above, it
suffices to prove that
$$
(B^{-1/2}AB^{-1/2})^\alpha\le B^{-1}
\ \,\implies\ \,A^{\alpha\over z}\le B^{\alpha-1\over z},
$$
or letting $C:=B^{-1/2}AB^{-1/2}>0$, we may prove that
\begin{equation}\label{F-2.9}
C^\alpha\le B^{-1}
\ \,\implies\ \,(B^{1/2}CB^{1/2})^{\alpha\over z}\le B^{\alpha-1\over z}.
\end{equation}
Now, assume that $\alpha>1$ and $z\ge\max\{\alpha/2,\alpha-1\}$. Note by \eqref{F-2.1} that
if once $Q_{\alpha,z}(A,B)\prec_{\log}P_\alpha(A,B)$ holds for $z=z_0$ with some $z_0>0$, then
the same does for all $z\ge z_0$. Hence we may further assume that $z\le\alpha$. If
$C^\alpha\le B^{-1}$, then $B\le C^{-\alpha}$ so that $C^{1/2}BC^{1/2}\le C^{1-\alpha}$. Since
$0\le{\alpha\over z}-1\le1$, we have
$$
(C^{1/2}BC^{1/2})^{{\alpha\over z}-1}\le(C^{1-\alpha})^{{\alpha\over z}-1}
=C^{(1-\alpha)({\alpha\over z}-1)}.
$$
Since by \eqref{F-2.5},
$$
(B^{1/2}CB^{1/2})^{\alpha\over z}
=B^{1/2}C^{1/2}(C^{1/2}BC^{1/2})^{{\alpha\over z}-1}C^{1/2}B^{1/2},
$$
we have
$$
(B^{1/2}CB^{1/2})^{\alpha\over z}\le B^{1/2}C^{1+(1-\alpha)({\alpha\over z}-1)}B^{1/2}
=B^{1/2}(C^\alpha)^{1-{\alpha-1\over z}}B^{1/2}.
$$
Since the assumption on $\alpha,z$ implies that $0\le1-{\alpha-1\over z}<1$, we have
$$
(B^{1/2}CB^{1/2})^{\alpha\over z}\le B^{1/2}(B^{-1})^{1-{\alpha-1\over z}}B^{1/2}
=B^{\alpha-1\over z},
$$
proving \eqref{F-2.9}.
\end{proof}

It should be noted that the above proofs of \eqref{F-2.8} and \eqref{F-2.9} are more or less
similar to that of \cite[Theorem 3]{Ta} for the Furuta inequality with negative powers. 

\begin{remark}\label{R-2.2}\rm
When $\alpha=2$, since $P_2(A,B)=AB^{-1}A$ is unitarily equivalent to $B^{-1/2}A^2B^{-1/2}$,
$P_2(A,B)\prec_{\log}Q_{2,z}(A,B)$ is equivalent to
$$
B^{-1/2}A^2B^{-1/2}\prec_{\log}(B^{-{1\over2z}}A^{2\over z}B^{-{1\over2z}})^z.
$$
Assume that $AB\ne BA$. Then from  Araki's log-majorization \cite{Ar} and
\cite[Theorem 2.1]{Hi1} we see that the above log-majorization holds true if and only if
$0<z\le1$. Similarly, $Q_{2,z}(A,B)\prec_{\log}P_2(A,B)$ holds if and only if $z\ge1$. These
are of course consistent with (2) and (3) of Proposition \ref{P-2.1}.
\end{remark}

In particular, when $A$ is a projection, we have:

\begin{prop}\label{P-2.3}
Let $\alpha>1$ and $z>0$. Assume that $E,B\in\bM_n^+$ are such that $E$ is a projection, $B>0$
and $EB\ne BE$. Then:
\begin{itemize}
\item[\rm(a)] $P_\alpha(E,B)\prec_{\log}Q_{\alpha,z}(E,B)$ if and only if $z\le\alpha-1$.
\item[\rm(b)] $Q_{\alpha,z}(E,B)\prec_{\log}P_\alpha(E,B)$ if and only if $z\ge\alpha-1$.
\end{itemize}
\end{prop}

\begin{proof}
We write
\begin{align*}
P_\alpha(E,B)&=B^{1/2}(B^{-1/2}EB^{-1/2})^\alpha B^{1/2} \\
&=B^{1/2}B^{-1/2}E(EB^{-1}E)^{\alpha-1}EB^{-1/2}B^{1/2}=(EB^{-1}E)^{\alpha-1}.
\end{align*}
On the other hand, $Q_{\alpha,z}(E,B)=(B^{1-\alpha\over2z}EB^{1-\alpha\over2z})^z$ is
unitarily equivalent to $(EB^{1-\alpha\over2z}E)^z$. Hence
$P_\alpha(E,B)\prec_{\log}Q_{\alpha,z}(E,B)$ is equivalent to
$$
(EB^{-1}E)^{\alpha-1}\prec_{\log}(EB^{1-\alpha\over2z}E)^z.
$$
From \cite[Theorem 2.1]{Hi1}
we see that this holds if and only if $\alpha-1\ge z$. Similarly,
$Q_{\alpha,z}(E,B)\prec_{\log}P_\alpha(E,B)$ holds if and only if $z\ge\alpha-1$. 
\end{proof}

\section{Log-majorization (Part 2)}

Our final goal is to completely determine the regions of $\{(\alpha,z):\alpha,z>0\}$ for which
$P_\alpha\prec_{\log}Q_{\alpha,z}$ holds, or $Q_{\alpha,z}\prec_{\log}P_\alpha$ holds, or
neither holds true, respectively. The next step to the goal is to find a region in $\alpha\ge2$
where $P_\alpha\prec_{\log}Q_{\alpha,z}$ holds true. Since $P_2\prec_{\log}Q_{2,z}$ holds
if and only if $0<z\le1$ (see Remark \ref{R-2.2}), it would be reasonable to conjecture that
there is a region in $\alpha\ge2$ touching $\{(2,z):0<z\le1\}$ where
$P_\alpha\prec_{\log}Q_{\alpha,z}$ holds.

We show the next log-majorization result by elaborating the anti-symmetric tensor power
technique. The proof reveals essentially similar features to those of \cite[Theorem 4.1]{AH}
and \cite[Theorem 2.1]{AHO}.

\begin{prop}\label{P-3.1}
Assume that $\alpha\ge2$ and $0<z\le\alpha/2$. Then for every $A,B\in\bM_n^+$ with $B>0$,
$$
P_\alpha(A,B)\prec_{\log}Q_{\alpha,z}(A,B).
$$
\end{prop}

\begin{proof}
By continuity we may assume that both $A,B$ are positive definite. Assume that $\alpha\ge2$
and $0<z\le\alpha/2$. Due to the anti-symmetric tensor power technique and the positive
homogeneity in $A,B$ of $P_\alpha$ and $Q_{\alpha,z}$, it suffices to prove that
$$
A^{\alpha\over z}\le B^{\alpha-1\over z}\ \,\implies\ \,P_\alpha(A,B)\le I.
$$
So assume that $A^{\alpha\over z}\le B^{\alpha-1\over z}$. We divide the proof into two cases.
First, assume that $2m\le\alpha\le2m+1$ for some $m\in\bN$, so write $\alpha=2m+\lambda$
with $0\le\lambda\le1$. Note that
\begin{align*}
P_\alpha(A,B)
&=B^{1/2}(B^{-1/2}AB^{-1/2})^m(B^{-1/2}AB^{-1/2})^\lambda(B^{-1/2}AB^{-1/2})^mB^{1/2} \\
&=(AB^{-1})^m(B\,\#_\lambda\,A)(B^{-1}A)^m.
\end{align*}
Since $0<z/\alpha\le1/2$, we have $A\le B^{\alpha-1\over\alpha}$ and hence
$$
B\,\#_\lambda\,A\le B\,\#_\lambda\,B^{\alpha-1\over\alpha}
=B^{1-\lambda}B^{(\alpha-1)\lambda\over\alpha}
=B^{\alpha-\lambda\over\alpha}.
$$
Therefore,
$$
P_\alpha(A,B)\le(AB^{-1})^{m-1}AB^{-{\alpha+\lambda\over\alpha}}A(B^{-1}A)^{m-1}.
$$
Since
$$
{(\alpha+\lambda)z\over\alpha(\alpha-1)}\le{\alpha+\lambda\over2(\alpha-1)}\le1
$$
thanks to $\alpha=2m+\lambda\ge2+\lambda$, we have
$$
B^{-{\alpha+\lambda\over\alpha}}
=(B^{\alpha-1\over z})^{-{(\alpha+\lambda)z\over\alpha(\alpha-1)}}
\le(A^{\alpha\over z})^{-{(\alpha+\lambda)z\over\alpha(\alpha-1)}}
=A^{-{\alpha+\lambda\over\alpha-1}},
$$
so that
$$
P_\alpha(A,B)\le(AB^{-1})^{m-1}A^{\alpha-2-\lambda\over\alpha-1}(B^{-1}A)^{m-1}.
$$
Since
$$
{(\alpha-2-\lambda)z\over\alpha(\alpha-1)}\le{\alpha-2-\lambda\over2(\alpha-1)}\le1,
$$
we have
$$
A^{\alpha-2-\lambda\over\alpha-1}
=(A^{\alpha\over z})^{(\alpha-2-\lambda)z\over\alpha(\alpha-1)}
\le(B^{\alpha-1\over z})^{(\alpha-2-\lambda)z\over\alpha(\alpha-1)}
=B^{\alpha-2-\lambda\over\alpha},
$$
so that
$$
P_\alpha(A,B)\le(AB^{-1})^{m-2}AB^{-{\alpha+2+\lambda\over\alpha}}A(B^{-1}A)^{m-2}.
$$
The above argument can be repeated to see that for $k=0,1,\dots,m-1$,
$$
P_\alpha(A,B)\le(AB^{-1})^{m-1-k}AB^{-{\alpha+2k+\lambda\over\alpha}}A(B^{-1}A)^{m-1-k},
$$
and hence
$$
P_\alpha(A,B)\le AB^{-{\alpha+2m-2+\lambda\over\alpha}}A=AB^{-{2(\alpha-1)\over\alpha}}A.
$$
Finally, since $2z/\alpha\le1$, we have
$$
B^{-{2(\alpha-1)\over\alpha}}\le(B^{\alpha-1\over z})^{-{2z\over\alpha}}
\le(A^{\alpha\over z})^{-{2z\over\alpha}}=A^{-2},
$$
so that $P_\alpha(A,B)\le I$.

Secondly, assume that $2m+1<\alpha<2m+2$ for some $m\in\bN$, so write $\alpha=2m+2-\lambda$
with $0<\lambda<1$. Note that
\begin{align*}
P_\alpha(A,B)
&=B^{1/2}(B^{-1/2}AB^{-1/2})^{m+1}(B^{-1/2}AB^{-1/2})^{-\lambda}
(B^{-1/2}AB^{-1/2})^{m+1}B^{1/2} \\
&=(AB^{-1})^mAB^{-1/2}(B^{-1/2}AB^{-1/2})^{-\lambda}B^{-1/2}A(B^{-1}A)^m \\
&=(AB^{-1})^mA(B\,\#_\lambda\,A)^{-1}A(B^{-1}A)^m.
\end{align*}
Since $0\le z/(\alpha-1)\le1$, we have $B\ge A^{\alpha\over\alpha-1}$ so that
$$
(B\,\#_\lambda\,A)^{-1}\le\bigl(A^{\alpha\over\alpha-1}\,\#_\lambda\,A\bigr)^{-1}
=\bigl(A^{\alpha(1-\lambda)\over\alpha-1}A^\lambda\bigr)^{-1}
=A^{-{\alpha-\lambda\over\alpha-1}}.
$$
Therefore,
$$
P_\alpha(A,B)\le(AB^{-1})^mA^{\alpha-2+\lambda\over\alpha-1}(B^{-1}A)^m.
$$
Since
$$
{(\alpha-2+\lambda)z\over\alpha(\alpha-1)}\le{\alpha-2+\lambda\over2(\alpha-1)}\le1,
$$
we have
$$
A^{\alpha-2+\lambda\over\alpha-1}
=(A^{\alpha\over z})^{(\alpha-2+\lambda)z\over\alpha(\alpha-1)}
\le(B^{\alpha-1\over z})^{(\alpha-2+\lambda)z\over\alpha(\alpha-1)}
=B^{\alpha-2+\lambda\over\alpha},
$$
so that
$$
P_\alpha(A,B)\le(AB^{-1})^{m-1}AB^{-{\alpha+2-\lambda\over\alpha}}A(B^{-1}A)^{m-1}.
$$
Since
$$
{(\alpha+2-\lambda)z\over\alpha(\alpha-1)}\le{\alpha+2-\lambda\over2(\alpha-1)}\le1,
$$
we have
$$
B^{-{\alpha+2-\lambda\over\alpha}}
=(B^{\alpha-1\over z})^{-{(\alpha+2-\lambda)z\over\alpha(\alpha-1)}}
\le(A^{\alpha\over z})^{-{(\alpha+2-\lambda)z\over\alpha(\alpha-1)}}
=A^{-{\alpha+2-\lambda\over\alpha-1}},
$$
so that
$$
P_\alpha(A,B)\le(AB^{-1})^{m-1}A^{\alpha-4+\lambda\over\alpha-1}(B^{-1}A)^{m-1}.
$$
Repeating the above argument we have
$$
P_\alpha(A,B)\le(AB^{-1})A^{\alpha-2m+\lambda\over\alpha-1}(B^{-1}A)
=AB^{-1}A^{2\over\alpha-1}B^{-1}A.
$$
Since
$$
A^{2\over\alpha-1}=(A^{\alpha\over z})^{2z\over\alpha(\alpha-1)}
\le(B^{\alpha-1\over z})^{2z\over\alpha(\alpha-1)}=B^{2\over\alpha}
$$
and
$$
B^{-{2(\alpha-1)\over\alpha}}=(B^{\alpha-1\over z})^{-{2z\over\alpha}}
\le(A^{\alpha\over z})^{-{2z\over\alpha}}=A^{-2},
$$
we finally have $P_\alpha(A,B)\le AB^{-{2(\alpha-1)\over\alpha}}A\le I$.
\end{proof}

Now, Theorem \ref{T-1.1} stated in the Introduction is proved from the log-majorization
results between $P_\alpha$ and $Q_{\alpha,z}$ obtained so far in Propositions \ref{P-2.1} and
\ref{P-3.1}.

We note that some discussions involving $P_\alpha$, $Q_\alpha=Q_{\alpha,1}$ and
$\widetilde Q_\alpha=Q_{\alpha,\alpha}$ were recently given in \cite[Sect.\ 5]{BJL}.

\section{Main theorem}

In this section we prove that Theorem \ref{T-1.1} is best possible with regard to the
assumptions on the parameters $\alpha,z$.

Assume that $\alpha>1$. For each $x,y>0$ and $\theta\in\bR$ define $2\times2$ positive
definite matrices
\begin{align}\label{F-4.1}
A_\theta:=\begin{bmatrix}\cos\theta&-\sin\theta\\\sin\theta&\cos\theta\end{bmatrix}
\begin{bmatrix}1&0\\0&x\end{bmatrix}
\begin{bmatrix}\cos\theta&\sin\theta\\-\sin\theta&\cos\theta\end{bmatrix},\qquad
B:=\begin{bmatrix}1&0\\0&y\end{bmatrix}.
\end{align}
We write
\begin{align*}
B^{-1/2}A_\theta B^{-1/2}
&=\begin{bmatrix}1+(x-1)\sin^2\theta&(1-x)y^{-1/2}{\sin2\theta\over2}\\
(1-x)y^{-1/2}{\sin2\theta\over2}&xy^{-1}+(1-x)y^{-1}\sin^2\theta\end{bmatrix} \\
&=G+\theta H+\theta^2K+o(\theta^2),
\end{align*}
where
$$
G:=\begin{bmatrix}1&0\\0&xy^{-1}\end{bmatrix},\quad
H:=\begin{bmatrix}0&(1-x)y^{-1/2}\\(1-x)y^{-1/2}&0\end{bmatrix},\quad
K:=\begin{bmatrix}x-1&0\\0&(1-x)y^{-1}\end{bmatrix},
$$
and $o(\theta^2)$ denotes a small value such that $o(\theta^2)/\theta^2\to0$ as $\theta\to0$.
We apply the Taylor formula with Fr\'echet derivatives (see e.g., \cite[Theorem 2.3.1]{Hi1})
to obtain
$$
(B^{-1/2}A_\theta B^{-1/2})^\alpha
=G^\alpha+D(x^\alpha)(G)(\theta H+\theta^2 K)
+{1\over2}D^2(x^\alpha)(G)(\theta H,\theta H)+o(\theta^2),
$$
where the second and the third terms in the right-hand side are the first and the second
Fr\'echet derivatives of $X\mapsto X^\alpha$ ($X\in\bM_2^+$, $X>0$) at $G$, respectively. By
Daleckii and Krein's derivative formula (see \cite[Theorem V.3.3]{Bh},
\cite[Theorem 2.3.1]{Hi1}) we have
\begin{align*}
&D(x^\alpha)(G)(\theta H+\theta^2 K) \\
&\quad=\begin{bmatrix}(x^\alpha)^{[1]}(1,1)&(x^\alpha)^{[1]}(1,xy^{-1})\\
(x^\alpha)^{[1]}(1,xy^{-1})&(x^\alpha)^{[1]}(xy^{-1},xy^{-1})\end{bmatrix}
\circ(\theta H+\theta^2 K) \\
&\quad=\begin{bmatrix}\alpha&{1-(xy^{-1})^\alpha\over1-xy^{-1}}\\
{1-(xy^{-1})^\alpha\over1-xy^{-1}}&\alpha(xy^{-1})^{\alpha-1}\end{bmatrix}
\circ(\theta H+\theta^2 K) \\
&\quad=\theta\begin{bmatrix}
0&{x^\alpha-y^\alpha\over x-y}(1-x)y^{{1\over2}-\alpha}\\
{x^\alpha-y^\alpha\over x-y}(1-x)y^{{1\over2}-\alpha}&0\end{bmatrix}
+\theta^2\begin{bmatrix}\alpha(x-1)&0\\
0&\alpha(1-x)x^{\alpha-1}y^{-\alpha}\end{bmatrix},
\end{align*}
where $(x^\alpha)^{[1]}$ denotes the first divided difference of $x^\alpha$ and $\circ$ means
the Schur (or Hadamard) product. For the second divided difference of $x^\alpha$ we compute
\begin{align*}
(x^\alpha)^{[2]}(1,1,xy^{-1})
&={\alpha-1-\alpha xy^{-1}+x^\alpha y^{-\alpha}\over(1-xy^{-1})^2}
={y\{(\alpha-1)y-\alpha x+x^\alpha y^{1-\alpha}\}\over(x-y)^2}, \\
(x^\alpha)^{[2]}(1,xy^{-1},xy^{-1})
&={1-\alpha x^{\alpha-1}y^{1-\alpha}+(\alpha-1)x^\alpha y^{-\alpha}\over(1-xy^{-1})^2} \\
&={y\{y-\alpha x^{\alpha-1}y^{2-\alpha}+(\alpha-1)x^\alpha y^{1-\alpha}\}\over(x-y)^2},
\end{align*}
and hence we have
\begin{align*}
{1\over2}D^2(x^\alpha)(G)(\theta H,\theta H)=\theta^2\begin{bmatrix}
{(x-1)^2\{(\alpha-1)y-\alpha x+x^\alpha y^{1-\alpha}\}\over(x-y)^2}&0\\
0&{(x-1)^2\{y-\alpha x^{\alpha-1}y^{2-\alpha}+(\alpha-1)x^\alpha y^{1-\alpha}\}\over(x-y)^2}
\end{bmatrix}.
\end{align*}
(In the above computation we have assumed that $x\ne y$.) Therefore, it
follows that
$$
(B^{-1/2}A_\theta B^{-1/2})^\alpha
=\begin{bmatrix}1+s_\alpha^{(1)}\theta^2&s_\alpha^{(3)}\theta\\
s_\alpha^{(3)}\theta&x^\alpha y^{-\alpha}+s_\alpha^{(2)}\theta^2\end{bmatrix}+o(\theta^2),
$$
where
\begin{align*}
s_\alpha^{(1)}&:=\alpha(x-1)
+{(x-1)^2\{(\alpha-1)y-\alpha x+x^\alpha y^{1-\alpha}\}\over(x-y)^2}, \\
s_\alpha^{(2)}&:=\alpha x^{\alpha-1}(1-x)y^{-\alpha}
+{(x-1)^2\{y-\alpha x^{\alpha-1}y^{2-\alpha}+(\alpha-1)x^\alpha y^{1-\alpha}\}\over(x-y)^2}.
\end{align*}
(The form of $s_\alpha^{(3)}$ is not written down here since it is unnecessary in the
computation below.) We hence arrive at
\begin{align}\label{F-4.2}
\Tr P_\alpha(A_\theta,B)
=1+x^\alpha y^{1-\alpha}+\bigl(s_\alpha^{(1)}+s_\alpha^{(2)}y\bigr)\theta^2+o(\theta^2).
\end{align}

Next, we write
\begin{align*}
B^{1-\alpha\over2z}A_\theta^{\alpha\over z}B^{1-\alpha\over z}
&=\begin{bmatrix}1+(x^{\alpha\over z}-1)\sin^2\theta
&(1-x^{\alpha\over z})y^{1-\alpha\over2z}{\sin2\theta\over2}\\
(1-x^{\alpha\over z})y^{1-\alpha\over2z}{\sin2\theta\over2}
&x^{\alpha\over z}y^{1-\alpha\over z}+(1-x^{\alpha\over z})y^{1-\alpha\over z}\sin^2\theta
\end{bmatrix} \\
&=\begin{bmatrix}1+(x^{\alpha\over z}-1)\theta^2
&(1-x^{\alpha\over z})y^{1-\alpha\over2z}\theta\\
(1-x^{\alpha\over z})y^{1-\alpha\over2z}\theta
&x^{\alpha\over z}y^{1-\alpha\over z}+(1-x^{\alpha\over z})y^{1-\alpha\over z}\theta^2
\end{bmatrix}+o(\theta^2).
\end{align*}
Since
$$
\det\Bigl(tI-B^{1-\alpha\over2z}A_\theta^{\alpha\over z}B^{1-\alpha\over z}\Bigr)
=t^2-\bigl\{1+x^{\alpha\over z}y^{1-\alpha\over z}
+(x^{\alpha\over z}-1)(1-y^{1-\alpha\over z})\theta^2\bigr\}t
+x^{\alpha\over z}y^{1-\alpha\over z}+o(\theta^2),
$$
the eigenvalues of $B^{1-\alpha\over2z}A_\theta^{\alpha\over z}B^{1-\alpha\over z}$ are
\begin{align*}
t_{\alpha,z,\theta}^\pm&={1\over2}\biggl[1+x^{\alpha\over z}y^{1-\alpha\over z}
+(x^{\alpha\over z}-1)(1-y^{1-\alpha\over z})\theta^2 \\
&\qquad\pm\sqrt{(1-x^{\alpha\over z}y^{1-\alpha\over z})^2
+2(x^{\alpha\over z}-1)(1-y^{1-\alpha\over z})
(1+x^{\alpha\over z}y^{1-\alpha\over z})\theta^2}\biggr]+o(\theta^2).
\end{align*}
Assuming that $1-x^{\alpha\over z}y^{1-\alpha\over z}>0$ (this is the case when we let
$y\to\infty$ for any fixed $x>0$), we have
\begin{align*}
t_{\alpha,z,\theta}^\pm&={1\over2}\biggl[1+x^{\alpha\over z}y^{1-\alpha\over z}
+(x^{\alpha\over z}-1)(1-y^{1-\alpha\over z})\theta^2 \\
&\qquad\pm\biggl\{1-x^{\alpha\over z}y^{1-\alpha\over z}
+{(x^{\alpha\over z}-1)(1-y^{1-\alpha\over z})
(1+x^{\alpha\over z}y^{1-\alpha\over z})\over
1-x^{\alpha\over z}y^{1-\alpha\over z}}\theta^2\biggr\}\biggr]+o(\theta^2),
\end{align*}
so that
\begin{align*}
t_{\alpha,z,\theta}^+&=1+{(x^{\alpha\over z}-1)(1-y^{1-\alpha\over z})\over
1-x^{\alpha\over z}y^{1-\alpha\over z}}\theta^2+o(\theta^2), \\
t_{\alpha,z,\theta}^-&=x^{\alpha\over z}y^{1-\alpha\over z}
\biggl\{1-{(x^{\alpha\over z}-1)(1-y^{1-\alpha\over z})
\over1-x^{\alpha\over z}y^{1-\alpha\over z}}\theta^2\biggr\}+o(\theta^2).
\end{align*}
Therefore, we have
\begin{align}
\Tr Q_{\alpha,z}(A_\theta,B)&=(t_{\alpha,z,\theta}^+)^z+(t_{\alpha,z,\theta}^-)^z \nonumber\\
&=1+z{(x^{\alpha\over z}-1)(1-y^{1-\alpha\over z})\over
1-x^{\alpha\over z}y^{1-\alpha\over z}}\theta^2 \nonumber\\
&\qquad+x^\alpha y^{1-\alpha}\biggl\{1-z{(x^{\alpha\over z}-1)(1-y^{1-\alpha\over z})
\over1-x^{\alpha\over z}y^{1-\alpha\over z}}\theta^2\biggr\}+o(\theta^2) \nonumber\\
&=1+x^\alpha y^{1-\alpha}+z{(x^{\alpha\over z}-1)
(1-y^{1-\alpha\over z})(1-x^\alpha y^{1-\alpha})\over
1-x^{\alpha\over z}y^{1-\alpha\over z}}\theta^2+o(\theta^2). \label{F-4.3}
\end{align}

Now, suppose that $Q_{\alpha,z}(A_\theta,B)\prec_{\log}P_\alpha(A_\theta,B)$ holds for all
$\theta\ne0$. Then we must have $\Tr Q_{\alpha,z}(A_\theta,B)\le\Tr P_\alpha(A_\theta,B)$.
(Since $\det Q_{\alpha,z}(A_\theta,B)=\det P_\alpha(A_\theta,B)$,
$Q_{\alpha,z}(A_\theta,B)\prec_{\log}P_\alpha(A_\theta,B)$ is indeed equivalent to
$\Tr Q_{\alpha,z}(A_\theta,B)\le\Tr P_\alpha(A_\theta,B)$ in the $2\times2$ case here.)
So by \eqref{F-4.2} and \eqref{F-4.3} it follows that
$$
z{(x^{\alpha\over z}-1)
(1-y^{1-\alpha\over z})(1-x^\alpha y^{1-\alpha})\over
1-x^{\alpha\over z}y^{1-\alpha\over z}}
\le s_\alpha^{(1)}+s_\alpha^{(2)}y.
$$
For any $x>0$, let $y\to\infty$; then the above left-hand side converges to
$z(x^{\alpha\over z}-1)$, while $s_\alpha^{(1)}\to\alpha(x-1)$ and $s_\alpha^{(2)}\to(x-1)^2$
thanks to $\alpha>1$. Hence we must have for every $x>0$,
$$
z(x^{\alpha\over z}-1)\le\alpha(x-1)+(x-1)^2.
$$
Letting $x\searrow0$ gives $-z\le-\alpha+1$, i.e., $z\ge\alpha-1$. Moreover, for any $x>1$,
$$
z{x^{\alpha\over z}-1\over x-1}\le x+\alpha-1,
$$
which holds true only when $\alpha/z\le2$, i.e., $z\ge\alpha/2$. On the other hand, suppose
that $P_\alpha(A_\theta,B)\prec_{\log}Q_{\alpha,z}(A_\theta,B)$ holds for all $\theta\ne0$.
Then, similarly to the above case, $z\le\alpha-1$ and $z\le\alpha/2$ must follow.

Thus, combining the above discussions with Theorem \ref{T-1.1} proves our main theorem as
follows:

\begin{thm}\label{T-4.1}
Let $\alpha,z>0$ with $\alpha\ne1$.
\begin{itemize}
\item[\rm(a)] The following conditions are equivalent:
\begin{itemize}
\item[\rm(i)] $P_\alpha(A,B)\prec_{\log}Q_{\alpha,z}(A,B)$ for every $A,B\in\bM_n^+$,
$n\in\bN$, with $B>0$;
\item[\rm(ii)] $\Tr P_\alpha(A,B)\le\Tr Q_{\alpha,z}(A,B)$ for every $A,B\in\bM_n^+$,
$n\in\bN$, with $B>0$;
\item[\rm(iii)] $P_\alpha(A,B)\prec_{\log}Q_{\alpha,z}(A,B)$ for every $A,B\in\bM_2^+$ with
$A,B>0$;
\item[\rm(iv)] either $0<\alpha<1$ and $z>0$ is arbitrary, or $\alpha>1$ and
$0<z\le\min\{\alpha/2,\alpha-1\}$.
\end{itemize}
\item[\rm(b)] The following conditions are equivalent:
\begin{itemize}
\item[\rm(i)$'$] $Q_{\alpha,z}(A,B)\prec_{\log}P_\alpha(A,B)$ for every $A,B\in\bM_n^+$,
$n\in\bN$, with $B>0$;
\item[\rm(ii)$'$] $\Tr Q_{\alpha,z}(A,B)\le\Tr P_\alpha(A,B)$ for every $A,B\in\bM_n^+$,
$n\in\bN$, with $B>0$;
\item[\rm(iii)$'$] $Q_{\alpha,z}(A,B)\prec_{\log}P_\alpha(A,B)$ for every $A,B\in\bM_2^+$ with
$A,B>0$;
\item[\rm(iv)$'$] $\alpha>1$ and $z\ge\max\{\alpha/2,\alpha-1\}$.
\end{itemize}
\end{itemize}
\end{thm}

The theorem says that neither $P_\alpha(A,B)\prec_{\log}Q_{\alpha,z}(A,B)$ nor
$Q_{\alpha,z}(A,B)\prec_{\log}P_\alpha(A,B)$ holds in general in the regions of
$1<\alpha<2$ and $\alpha-1<z<\alpha/2$ and of $\alpha>2$ and $\alpha/2<z<\alpha-1$.

\section{Further extension}

For $A,B\in\bM_n^+$ with $B>0$, taking the expression
$Q_{\alpha,z}(A,B)=Q_\alpha(A^{1/z},B^{1/z})^z$ into account, we may define the two-parameter
extension of $P_\alpha$ as
$$
P_{\alpha,r}(A,B):=P_\alpha(A^{1/r},B^{1/r})^r
=\{B^{1/2r}(B^{-1/2r}A^{1/r}B^{-1/2r})^\alpha B^{1/2r}\}^r,\qquad\alpha,r>0.
$$
The log-majorization in \cite{AH} says that when $0<\alpha<1$,
\begin{align}\label{F-5.1}
P_{\alpha,r}(A,B)\prec_{\log}P_{\alpha,r'}(A,B)\qquad
\mbox{if $0<r\le r'$}.
\end{align}
For every $A,B>0$ and $\alpha>0$, note by \eqref{F-2.7} that
$P_\alpha(A,B)=A^{1/2}(A^{-1/2}BA^{-1/2})^\beta A^{1/2}$ where $\beta:=1-\alpha$ (the
right-hand side is often denoted by $A\natural_\beta B$ when $\beta\not\in[0,1]$ instead of
$A\#_\beta B$ for $\beta\in[0,1]$). Thus, the
log-majorization recently obtained in \cite[Theorem 3.1]{KS} is rephrased as follows: When
$1<\alpha\le2$, for every $A,B\in\bM_n^+$ with $B>0$,
\begin{align}\label{F-5.2}
P_{\alpha,r}(A,B)\prec_{\log}P_{\alpha,r'}(A,B)\qquad
\mbox{if $0<r'\le r$}.
\end{align}
In particular, when $\alpha=2$, this reduces to Araki's log-majorization (see Remark
\ref{R-2.2}). 

For each $\alpha,r,z>0$ with $\alpha\ne1$, since it is easy
to see that $P_{\alpha,r}(A,B)\prec_{\log}Q_{\alpha,z}(A,B)$ (resp.,
$Q_{\alpha,z}(A,B)\prec_{\log}P_{\alpha,r}(A,B)$) for every $A,B\in\bM_n^+$ with
$B>0$ if and only if $P_\alpha(A,B)\prec_{\log}Q_{\alpha,z/r}(A,B)$ (resp.,
$Q_{\alpha,z/r}(A,B)\prec_{\log}P_\alpha(A,B)$) for every $A,B\in\bM_n^+$ with $B>0$. Thus,
we can extend Theorem \ref{T-4.1} in the following way:

\begin{prop}\label{P-5.1}
Let $\alpha,r,z>0$ with $\alpha\ne1$
\begin{itemize}
\item[\rm(a)] The following conditions are equivalent:
\begin{itemize}
\item[\rm(i)] $P_{\alpha,r}(A,B)\prec_{\log}Q_{\alpha,z}(A,B)$ for every $A,B\in\bM_n^+$,
$n\in\bN$, with $B>0$;
\item[\rm(ii)] either $0<\alpha<1$ and $r,z>0$ are arbitrary, or $\alpha>1$ and
$0<z/r\le\min\{\alpha/2,\alpha-1\}$.
\end{itemize}
\item[\rm(b)] The following conditions are equivalent:
\begin{itemize}
\item[\rm(i)$'$] $Q_{\alpha,z}(A,B)\prec_{\log}P_{\alpha,r}(A,B)$ for every $A,B\in\bM_n^+$,
$n\in\bN$, with $B>0$;
\item[\rm(ii)$'$] $\alpha>1$ and $z/r\ge\max\{\alpha/2,\alpha-1\}$.
\end{itemize}
\end{itemize}
\end{prop} 

Although Proposition \ref{P-5.1} is just a slight modification of Theorem \ref{T-4.1}, it can
be used to show the following log-majorization supplementary to \eqref{F-5.1} and \eqref{F-5.2}:

\begin{cor}\label{C-5.2}
Assume that $\alpha\ge2$. For every $A,B\in\bM_n^+$ with $B>0$,
$$
P_{\alpha,r}(A,B)\prec_{\log}P_{\alpha,r'}(A,B)\qquad
\mbox{if $0<r'\le{\alpha\over2(\alpha-1)}r$}.
$$
Hence $P_{\alpha,r}(A,B)\prec_{\log}P_{\alpha,r'}(A,B)$ for all $\alpha\ge2$ if $0<r'\le r/2$.
\end{cor}

\begin{proof}
Let $\alpha\ge2$. By Proposition \ref{P-5.1}\,(a) we have
$$
P_{\alpha,r}(A,B)\prec_{\log}Q_{\alpha,r\alpha/2}(A,B).
$$
Since $(r\alpha/2)/r'\ge\alpha-1$, Proposition \ref{P-5.1}\,(b) implies that
$$
Q_{\alpha,r\alpha/2}(A,B)\prec_{\log}P_{\alpha,r'}(A,B),
$$
so that the asserted log-majorization follows.
\end{proof}

\begin{problem}\label{P-5.3}\rm
Although the assumption $\beta=1-\alpha\in[-1,0)$ (or $1<\alpha\le2$) seems essential in the
proof of \eqref{F-5.2} in \cite{KS}, it is unknown whether \eqref{F-5.2} holds true even for
$\alpha>2$ (i.e., the bound $\alpha/2(\alpha-1)$ in the corollary can be removed) or not. For
example, when $\alpha={m+1}\in\bN$ with $m\in\bN$, $m\ge2$, noting $P_{m+1}(A,B)=(AB^{-1})^mA$
and replacing $B^{-1}$ with $B$ and $1/r$ with $r$, \eqref{F-5.2} is equivalent to the
following extended Araki's log-majorization for every $A,B\in\bM_n^+$:
\begin{align}\label{F-5.3}
((AB)^mA)^r\prec_{\log}(A^rB^r)^mA^r\qquad\mbox{if $r\ge1$},
\end{align}
which seems difficult to hold in general, while no counter-example is at the moment known to
us. But Corollary \ref{C-5.2} implies that $((AB)^mA)^r\prec_{\log}(A^rB^r)^mA^r$ for
$r\ge2m/(m+1)$. Here is a simple argument when $m=2$. For $m=2$, to prove \eqref{F-5.3}, it
suffices to show that for $0<p\le1$, $ABABA\le I$ $\implies$ $A^pB^pA^pB^pA^p\le I$. Assume
the left-hand inequality, i.e., $(A^{1/2}BA^{1/2})^2\le A^{-1}$; then $B\le A^{-3/2}$ and so
$A^p\le B^{-2p/3}$. Hence $A^pB^pA^pB^pA^p\le A^pB^{4p/3}A^p$. If $p\le3/4$, then
$B^{4p/3}\le A^{-2p}$ and so $A^pB^pA^pB^pA^p\le I$. Therefore,
$(ABABA)\prec_{\log}A^rB^rA^rB^rA^r$ if $r\ge4/3$, which is just the case $\alpha=3$ of
the corollary. The same argument works well when $\alpha=m+1$ for any $m\in\bN$,
$m\ge2$, proving directly the $\alpha=m+1$ case of the corollary.
\end{problem}

\section{Norm inequalities and their equality cases}

A norm $\|\cdot\|$ on $\bM_n$ is said to be \emph{unitarily invariant} if $\|UXV\|=\|X\|$ for
all $X\in\bM_n$ and all unitaries $U,V\in\bM_n$. We say (see \cite{Hi1}) that a unitarily
invariant norm $\|\cdot\|$ is \emph{strictly increasing} if for $X,Y\in\bM_n^+$, $X\le Y$ and
$\|X\|=\|Y\|$ imply $X=Y$. For example, the Schatten $p$-norm $\|X\|_p:=(\Tr|X|^p)^{1/p}$ is
strictly increasing when $1\le p<\infty$.

Theorem \ref{T-1.1} implies the following:

\begin{cor}\label{C-6.1}
Let $A,B\in\bM_n^+$ with $B>0$ and $\|\cdot\|$ be a unitarily invariant norm on $\bM_n$.
\begin{itemize}
\item[\rm(1)] If $0<\alpha<1$, then $\|P_\alpha(A,B)\|\le\|Q_{\alpha,z}(A,B)\|$ for all $z>0$.
\item[\rm(2)] If $\alpha>1$ and $0<z\le\min\{\alpha/2,\alpha-1\}$, then
$\|P_\alpha(A,B)\|\le\|Q_{\alpha,z}(A,B)\|$.
\item[\rm(3)] If $\alpha>1$ and $z\ge\max\{\alpha/2,\alpha-1\}$, then
$\|Q_{\alpha,z}(A,B)\|\le\|P_\alpha(A,B)\|$.
\end{itemize}
\end{cor}

\begin{remark}\label{R-6.2}\rm
The norm inequalities with negative power $\beta$ in \cite[Theorem 4.4]{KS} can be
rephrased as follows (by letting $\alpha=1-\beta$): When $A,B>0$, for every unitarily
invariant norm,
\begin{align*}
\|Q_{\alpha,z}(B,A)\|\le\|P_\alpha(B,A)\|\qquad
&\mbox{if $\alpha\in(1,2]$, $z\ge2$}, \\
\|P_\alpha(B,A)\|\le\|Q_{\alpha,1/2}(B,A)\|\le\|Q_{\alpha,z}(B,A)\|\qquad
&\mbox{if $\alpha\in[3/2,2]$, $0<z\le1/2$}.
\end{align*}
These inequalities are indeed included in (2) and (3) of Corollary \ref{C-6.1} (and
\eqref{F-2.1}).
\end{remark}

\begin{lemma}\label{L-6.3}
Assume that $\alpha>0$ and $\alpha\ne1$. Let $\|\cdot\|$ be a strictly increasing unitarily
invariant norm on $\bM_n$. If $\|Q_{\alpha,z}(A,B)\|=\|Q_{\alpha,z'}(A,B)\|$ for some $z,z'>0$
with $z\ne z'$, then $AB=BA$.
\end{lemma}

\begin{proof}
By \cite[Theorem 2.1]{Hi1} the assumed norm equality implies that $A^\alpha$ and $B^{1-\alpha}$
commute and hence $AB=BA$.
\end{proof}

Concerning the equality cases of the inequalities in Corollary \ref{C-6.1} we have:

\begin{prop}\label{P-6.4}
Let $\|\cdot\|$ be a strictly increasing unitarily invariant norm on $\bM_n$. Then we have
$AB=BA$ if $\|P_\alpha(A,B)\|=\|Q_{\alpha,z}(A,B)\|$ for some $\alpha,z$ satisfying one of the
following:
\begin{itemize}
\item[\rm(1)] $0<\alpha<1$ and $z>0$,
\item[\rm(2)] $\alpha>1$ and $0<z<\min\{\alpha/2,\alpha-1\}$,
\item[\rm(3)] $\alpha>1$ and $z>\max\{\alpha/2,\alpha-1\}$.
\end{itemize}
\end{prop}

\begin{proof}
(1)\enspace
Assume that $\|P_\alpha(A,B)\|=\|Q_{\alpha,z}(A,B)\|$ for some $\alpha,z$ in (1). Choose
$z'>z$. By \eqref{F-2.1} and Corollary \ref{C-6.1}\,(1) we have
$$
\|P_\alpha(A,B)\|=\|Q_{\alpha,z}(A,B)\|\ge\|Q_{\alpha,z'}(A,B)\|\ge\|P_\alpha(A,B)\|,
$$
implying $AB=BA$ by Lemma \ref{L-6.3}.

(2)\enspace
Assume that $\|P_\alpha(A,B)\|=\|Q_{\alpha,z}(A,B)\|$ for some $\alpha,z$ in (2). Choose $z'$
with $z<z'<\min\{\alpha/2,\alpha-1\}$. By \eqref{F-2.1} and Corollary \ref{C-6.1}\,(2),
$$
\|P_\alpha(A,B)\|=\|Q_{\alpha,z}(A,B)\|\ge\|Q_{\alpha,z'}(A,B)\|\ge\|P_\alpha(A,B)\|,
$$
implying $AB=BA$ by Lemma \ref{L-6.3}.

(3)\enspace
Assume that $\|P_\alpha(A,B)\|=\|Q_{\alpha,z}(A,B)\|$ for some $\alpha,z$ in (3). Choose $z'$
with $z>z'>\max\{\alpha/2,\alpha-1\}$. Then $AB=BA$ follows similarly to the proof for (2).
\end{proof}

\section{Logarithmic trace inequalities}

For every $p>0$ and every $A,B\in\bM_n^+$ with $B>0$, the logarithmic trace inequalities  
\begin{align}
{1\over p}\Tr A\log(B^{-p/2}A^pB^{-p/2})
&\le\Tr A(\log A-\log B) \nonumber\\
&\le{1\over p}\Tr A\log(A^{p/2}B^{-p}A^{p/2}). \label{F-7.1}
\end{align}
were shown in \cite{HP2}, and supplementary logarithmic trace inequalities were also in
\cite{AH}. In particular, the
latter inequality for $p=1$ was first proved in \cite{HP1}, giving the comparison between the
Umegaki relative entropy and the Belavkin-Staszewski relative entropy \cite{BS} (see Remark
\ref{R-8.5} below). Recall that this can readily be verified by taking the derivatives at
$\alpha=1$ of $\Tr P_\alpha(A,B)=\Tr A(A^{1/2}B^{-1}A^{1/2})^{\alpha-1}$ and
$\Tr Q_\alpha(A,B)=\Tr A^\alpha B^{1-\alpha}$ from Corollary \ref{C-6.1}. By the derivatives
at $\alpha=2$ we have more logarithmic trace inequalities in the following:

\begin{prop}
For every $A,B\in\bM_n^+$ with $B>0$,
\begin{align}
\Tr AB^{-1}A(\log A-\log B)
&\le\Tr A^{1/2}B^{-1}A^{3/2}\log(A^{1/2}B^{-1}A^{1/2}) \nonumber\\
&=\Tr B^{-1/2}A^2B^{-1/2}\log(B^{-1/2}AB^{-1/2}) \nonumber\\
&=\Tr A^{3/2}B^{-1}A^{1/2}\log(A^{1/2}B^{-1}A^{1/2}) \nonumber\\
&\le\Tr A^2B^{-1}(\log A-\log B). \label{F-7.2}
\end{align}
\end{prop}

\begin{proof}
To prove the inequalities and the equalities above, we may assume by continuity that $A>0$ as
well as $B>0$. The inequalities in the middle of \eqref{F-7.2} are easily verified as
\begin{align*}
\Tr A^{1/2}B^{-1}A^{3/2}\log(A^{1/2}B^{-1}A^{1/2})
&=\Tr A\log(A^{1/2}B^{-1}A^{1/2})\cdot A^{1/2}B^{-1}A^{1/2} \\
&=\Tr A^{3/2}B^{-1}A^{1/2}\log(A^{1/2}B^{-1}A^{1/2}) \\
&=\Tr A^{3/2}B^{-1/2}\log(B^{-1/2}AB^{-1/2})\cdot B^{-1/2}A^{1/2} \\
&=\Tr B^{-1/2}A^2B^{-1/2}\log(B^{-1/2}AB^{-1/2}),
\end{align*}
where we have used \eqref{F-2.6} for the third equality. To prove the inequalities, we use
Corollary \ref{C-6.1}\,(2) for $z=1$ to have
$$
\Tr P_\alpha(A,B)\le\Tr Q_\alpha(A,B)\qquad\mbox{for $\alpha\ge2$}.
$$
Since $\Tr P_2(A,B)=\Tr A^2B^{-1}=\Tr Q_2(A,B)$, if follows that
\begin{equation}\label{F-7.3}
{d\over d\alpha}\,\Tr P_\alpha(A,B)\bigg|_{\alpha=2}
\le{d\over d\alpha}\,\Tr Q_\alpha(A,B)\bigg|_{\alpha=2}.
\end{equation}
The left-hand side of \eqref{F-7.3} is
\begin{align*}
&\Tr B(B^{-1/2}AB^{-1/2})^2\log(B^{-1/2}AB^{-1/2}) \\
&\qquad=\Tr B^{1/2}AB^{-1}AB^{-/2}\log(B^{-1/2}AB^{-1/2}) \\
&\qquad=\Tr B^{1/2}AB^{-1}A^{1/2}\log(A^{1/2}B^{-1}A^{1/2})\cdot A^{1/2}B^{-1/2} \\
&\qquad=\Tr A^{3/2}B^{-1}A^{1/2}\log(A^{1/2}B^{-1}A^{1/2},
\end{align*}
where we have used \eqref{F-2.6} again for the second equality. On the other hand, the
right-hand side of \eqref{F-7.3} is
$$
\Tr A^2\log A\cdot B^{-1}-\Tr A^2B^{-1}\log B=\Tr A^2B^{-1}(\log A-\log B).
$$
Hence the latter inequality in \eqref{F-7.2} follows.

Next, set $C:=B^{-1/2}AB^{-1/2}$ so that $A=B^{1/2}CB^{1/2}$. Then
$$
\Tr B^{-1/2}A^2B^{-1/2}\log(B^{-1/2}AB^{-1/2})=\Tr CBC\log C
$$
and
\begin{align*}
&\Tr A^2B^{-1}(\log A-\log B) \\
&\qquad=\Tr B^{1/2}CBCB^{-1/2}\bigl(\log(B^{1/2}CB^{1/2})-\log B\bigr) \\
&\qquad=\Tr C^{1/2}BCB^{-1/2}\log(B^{1/2}CB^{1/2})\cdot B^{1/2}C^{1/2}-\Tr CBC\log B \\
&\qquad=\Tr C^{1/2}BC^{3/2}\log(C^{1/2}BC^{1/2})-\Tr CBC\log B.
\end{align*}
by \eqref{F-2.6} once again. Hence the latter inequality in \eqref{F-7.2} is rephrased as
$$
\Tr CBC(\log C+\log B)\le\Tr C^{1/2}BC^{3/2}\log(C^{1/2}BC^{1/2}).
$$
Replacing $C,B$ with $A,B^{-1}$, respectively, we have the first inequality in \eqref{F-7.2}.
\end{proof}

\begin{remark}\rm
It is obvious that if $A,B$ are commuting, then all the inequalities of \eqref{F-7.1} and
\eqref{F-7.2} become equality. In the converse direction, it is seen from
\cite[Theorem 5.1]{AH} and \cite[Theorem 4.1]{Hi1} that the equality case of the second
inequality of \eqref{F-7.1} (for some $p>0$) implies $AB=BA$. Here we note that if equality
holds in both inequalities of \eqref{F-7.2} then $AB=BA$. Indeed, the inequality between both
ends of \eqref{F-7.2} means that
\begin{equation}\label{F-3.4}
\Tr AB^{-1}A\log B^{-1}\le\Tr A^2B^{-1}\log B^{-1},
\end{equation}
which is considered as a kind of so-called \emph{gathering inequalities} (see, e.g., \cite{Bo}
and \cite{AHO}). To prove that the equality case of \eqref{F-3.4} implies $AB=BA$, we may
assume that $B$ is diagonal, so $B^{-1}=\diag(\lambda_1,\dots,\lambda_n)$. Then for
$A=[a_{ij}]_{i,j=1}^n$, equality in \eqref{F-3.4} means that
$$
\sum_{i,j=1}^n|a_{ij}|^2\lambda_i\log\lambda_j
=\sum_{i,j=1}^n|a_{ij}|^2\lambda_i\log\lambda_i,
$$
which is rewritten as
$$
\sum_{i,j=1}^n|a_{ij}|^2(\lambda_i-\lambda_j)(\log\lambda_i-\log\lambda_j)=0.
$$
Since $(\lambda_i-\lambda_j)(\log\lambda_i-\log\lambda_j)>0$ when $\lambda_i\ne\lambda_j$, we
must have $a_{ij}=0$ for all $i,j$ with $\lambda_i\ne\lambda_j$, implying $AB=BA$.

We may naturally conjecture that if either inequality of \eqref{F-7.2} holds with equality then
$AB=BA$.

\end{remark}

\section{Applications to R\'enyi divergences}

In this section we apply our log-majorization results to the relations between R\'enyi type
divergences $D_\alpha$, $\widetilde D_\alpha$, $\widehat D_\alpha$ and $D_{\alpha,z}$
defined in the Introduction.

The equivalences (ii)\,$\iff$\,(iv) and (ii)$'$\,$\iff$\,(iv)$'$ of Theorem \ref{T-4.1}
immediately yield the following:

\begin{cor}\label{C-8.1}
Let $\alpha,z>0$ with $\alpha\ne1$.
\begin{itemize}
\item[\rm(a)] The following conditions are equivalent:
\begin{itemize}
\item[\rm(i)] $\widehat D_\alpha(A\|B)\le D_{\alpha,z}(A\|B)$ for every $A,B\in\bM_n^+$,
$n\in\bN$, with $A\ne0$ and $B>0$;
\item[\rm(ii)] $\alpha>1$ and $z\le\min\{\alpha/2,\alpha-1\}$.
\end{itemize}
\item[\rm(b)] The following conditions are equivalent:
\begin{itemize}
\item[\rm(i)$'$] $D_{\alpha,z}(A\|B)\le\widehat D_\alpha(A\|B)$ for every $A,B\in\bM_n^+$,
$n\in\bN$, with $A\ne0$ and $B>0$;
\item[\rm(ii)$'$] either $0<\alpha<1$ and $z>0$ is arbitrary, or $\alpha>1$ and
$z\ge\max\{\alpha/2,\alpha-1\}$.
\end{itemize}
\end{itemize}
\end{cor}

Moreover, specializing Corollary \ref{C-6.1} to $z=1,\alpha$ and the trace-norm, we have:

\begin{cor}\label{C-8.2}
Let $A,B\in\bM_n^+$ with $A\ne0$ and $B>0$. If $0<\alpha\le2$ and $\alpha\ne1$ then
\begin{align}\label{F-8.1}
\widetilde D_\alpha(A\|B)\le D_\alpha(A\|B)\le\widehat D_\alpha(A\|B),
\end{align}
and if $\alpha\ge2$ then
$$
\widetilde D_\alpha(A\|B)\le\widehat D_\alpha(A\|B)\le D_\alpha(A\|B).
$$
\end{cor}

\begin{cor}\label{C-8.3}
Let $A,B\in\bM_n^+$ with $A\ne0$ and $B>0$. If some two of $D_\alpha(A\|B)$,
$\widetilde D_\alpha(A\|B)$, and $\widehat D_\alpha(A\|B)$ are equal for some
$\alpha\in(0,\infty)\setminus\{1,2\}$, then $AB=BA$.
\end{cor}

\begin{proof}
Note that $D_\alpha=\widehat D_\alpha$ means $\Tr Q_{\alpha,1}=\Tr P_\alpha$, and
$\widetilde D_\alpha=\widehat D_\alpha$ means $\Tr Q_{\alpha,\alpha}=\Tr P_\alpha$. Also,
note that if $1<\alpha<2$ then $\alpha>1>\max\{\alpha/2,\alpha-1\}$, and if $\alpha>2$ then
$1<\min\{\alpha/2,\alpha-1\}$ and $\alpha>\max\{\alpha/2,\alpha-1\}$. Hence by Proposition
\ref{P-6.4}, either equality of $D_\alpha(A\|B)=\widehat D_\alpha(A\|B)$ or
$\widetilde D_\alpha(A\|B)=\widehat D_\alpha(A\|B)$ implies $AB=BA$. Furthermore,
$D_\alpha(A\|B)=\widetilde D_\alpha(A\|B)$ implies $AB=BA$ by Lemma \ref{L-6.3}.
\end{proof}

\begin{remark}\label{R-8.4}\rm
In \cite{HM} we studied the \emph{standard $f$-divergence} $S_f(A\|B)$ and the \emph{maximal
$f$-divergence} $\widehat S_f(A\|B)$, which are defined as
\begin{align*}
S_f(A\|B)&:=\Tr B^{1/2}f(L_AR_{B^{-1}})(B^{1/2}), \\
\widehat S_f(A\|B)&:=\Tr P_f(A,B)
\end{align*}
for $A,B\in\bM_n^+$, $A,B>0$ (and extended to general $A,B\in\bM_n^+$ by convergences), where
$L_A$ is the left multiplication on $\bM_n$ by $A$ and $R_{B^{-1}}$ is the right
multiplication by $B^{-1}$. It is known \cite[Proposition 4.1]{HM} (see also \cite{Ma}) that
$$
S_f(A\|B)\le\widehat S_f(A\|B)
$$
holds whenever $f$ is an operator convex function on $(0,\infty)$. When $f(x)=-x^\alpha$ for
$0<\alpha<1$ or $f(x)=x^\alpha$ for $1<\alpha\le2$, this becomes the second inequality of
\eqref{F-8.1}. Corollaries \ref{C-8.2} and \ref{C-8.3} say that this is no longer true if $f$
is a general convex function on $(0,\infty)$. Furthermore, a special case of
\cite[Theorem 4.3]{HM} says that $D_\alpha(A\|B)=\widehat D_\alpha(A\|B)$ for some
$\alpha\in(0,2)\setminus\{1\}$ implies $AB=BA$, which is included in Corollary \ref{C-8.3}.
\end{remark}

\begin{remark}\label{R-8.5}\rm
Let $A,B\in\bM_n^+$ with $A\ne0$ and $B>0$ as above. It is well-known (and readily verified)
that
$$
\lim_{\alpha\to1}D_\alpha(A\|B)=D_1(A\|B):={D(A\|B)\over\Tr A},
$$
where $D(A\|B):=\Tr A(\log A-\log B)$, the \emph{Umegaki relative entropy}. It is also known
\cite{MDSFT} that
$$
\lim_{\alpha\to1}\widetilde D_\alpha(A\|B)=D_1(A\|B).
$$
On the other hand, we note that
$$
\lim_{\alpha\to1}\widehat D_\alpha(A\|B)
={1\over\Tr A}\,\Tr B^{1/2}AB^{-1/2}\log(B^{-1/2}AB^{-1/2})={D_\BS(A\|B)\over\Tr A},
$$
where $D_\BS(A\|B):=\Tr A\log(A^{1/2}B^{-1}A^{1/2})$, the \emph{Belavkin-Staszewski relative
entropy} \cite{BS} (see also \cite[Example 4.4]{HM}). By Corollary \ref{C-8.2} we have
$D(A\|B)\le D_\BS(A\|B)$, which was first obtained in \cite{HP1}.
\end{remark}

\subsection*{Acknowledgments}

This work was supported by JSPS KAKENHI Grant Number JP17K05266.


\begin{thebibliography}{99}

\bibitem{An} T. Ando, Majorization, doubly stochastic matrices, and
comparison of eigenvalues, {\it Linear Algebra Appl.} {\bf 118} (1989), 163--248.

\bibitem{AH}
T. Ando and F. Hiai, Log majorization and complementary Golden-Thompson type inequalities,
{\it Linear Algebra Appl.} {\bf 197/198} (1994), 113--131.

\bibitem{AHO}
T. Ando, F. Hiai and K. Okubo, Trace inequalities for multiple products of two matrices,
{\it Math. Ineq. Appl.} {\bf 3} (2000), 307--318.

\bibitem{Ar}
H. Araki, On an inequality of Lieb and Thirring, {\it Lett. Math. Phys.} {\bf 19}
(1990), 167--170.

\bibitem{AD}
K. M. R. Audenaert and N. Datta, $\alpha$-$z$-relative entropies,
{\it J. Math. Phys.} {\bf 56} (2015), 022202.

\bibitem{BS}
V. P. Belavkin and P. Staszewski, $C^*$-algebraic generalization of relative entropy and
entropy, {\it Ann. Inst. H. Poincar\'e Sect. A} {\bf 37} (1982), 51--58.

\bibitem{Bh}
R. Bhatia, {\it Matrix Analysis}, Springer-Verlag, New York, 1996.

\bibitem{BJL}
R. Bhatia, T. Jain and Y. Lim, Strong convexity of sandwiched entropy and related
optimization problems, {\it Rev. Math. Phys.}, to appear.

\bibitem{Bo}
J.-C. Bourin, Some inequalities for norms on matrices and operator,
{\it Linear Algebra Appl.} {\bf 292} (1999), 139--154.

\bibitem{EH}
E. Effros and F. Hansen, Non-commutative perspectives,
{\it Ann. Funct. Anal.} {\bf 5} (2014), 74--79.

\bibitem{Hi1}
F. Hiai, Equality cases in matrix norm inequalities of Golden-Thompson type,
{\it Linear and Multilinear Algebra} {\bf 36} (1994), 239--249.

\bibitem{Hi2}
F. Hiai, Matrix Analysis: Matrix Monotone Functions, Matrix Means, and Majorization,
{\it Interdisciplinary Information Sciences} {\bf 16} (2010), 139--248.

\bibitem{HM}
F. Hiai and M. Mosonyi, Different quantum $f$-divergences and the reversibility of
quantum operations, {\it Rev. Math. Phys.} {\bf 29} (2017), 1750023, 80 pp.

\bibitem{HP1}
F. Hiai and D. Petz,
The proper formula for relative entropy and its asymptotics in quantum probability,
{\it Comm. Math. Phys.} {\bf 143} (1991), 99--114.

\bibitem{HP2}
F. Hiai and D. Petz,
The Golden-Thompson trace inequality is complemented,
{\it Linear Algebra Appl.} {\bf 181} (1993), 153--185.

\bibitem{KS}
M. Kian and Y. Seo, Norm inequalities related to the matrix geometric mean of negative
power, {\it Sci. Math. Japon.}, Online, 2018.

\bibitem{KA}
F. Kubo and T. Ando, Means of positive linear operators, {\it Math. Ann.} {\bf 246} (1980),
205--224.

\bibitem{MOA}
A. W. Marshall, I. Olkin and B. C. Arnold,
{\it Inequalities: Theory of Majorization and Its Applications}, Second ed.,
Springer-Verlag, New York, 2011.

\bibitem{Ma}
K. Matsumoto, A new quantum version of $f$-divergence, arXiv:1311.4722, 2014.

\bibitem{MDSFT}
M. M\"uller-Lennert, F. Dupuis, O. Szehr, S. Fehr and M. Tomamichel,
On quantum R\'enyi entropies: A new generalization and some properties,
{\it J. Math. Phys.} {\bf 54} (2013), 122203.

\bibitem{Pe}
D. Petz, Quasi-entropies for finite quantum systems,
{\it Rep. Math. Phys.} {\bf 23} (1986), 57--65.

\bibitem{PW}
W. Pusz and S. L. Woronowicz, Functional calculus for sesquilinear forms and the purification map, {\it Rep. Math. Phys.} {\bf 8} (1975), 159--170.

\bibitem{Ta}
K. Tanahashi,
The Furuta inequality with negative powers, {\it Proc. Amer. Math. Soc.} {\bf 127} (1999),
1683--1692.

\bibitem{WWY}
M. M. Wilde, A. Winter and D. Yang,
Strong converse for the classical capacity of entanglement-breaking
and Hadamard channels via a sandwiched R\'enyi relative Entropy,
{\it Comm. Math. Phys.} {\bf 331} (2014), 593--622.

\end{thebibliography}
\end{document}